\documentclass[notitlepage,11pt,reqno]{amsart}
\usepackage[hmargin={1.0in, 1.0in}, vmargin={1.0in, 1.0in}]{geometry}
\usepackage{amssymb,amsmath,amsthm, thmtools, enumitem}
\usepackage{comment}
\usepackage{stmaryrd}
\usepackage{mathtools}
\usepackage[mathscr]{eucal}
\usepackage[breaklinks=true, linktocpage]{hyperref}%To get clickable links to papers
\usepackage{tikz-cd}
\usetikzlibrary{decorations.pathmorphing}
\usepackage{rotating}
\usepackage{xcolor}
\usepackage[
alphabetic,
backrefs,
msc-links,
nobysame,
lite,
]{amsrefs}

\DefineSimpleKey{bib}{primaryclass}{}
\DefineSimpleKey{bib}{archiveprefix}{}

\BibSpec{arXiv}{%
  +{}{\PrintAuthors}{author}
  +{,}{ \textit}{title}
  +{}{ \parenthesize}{date}
  +{,}{ preprint, arXiv: }{eprint}
}

\DeclareMathOperator{\Sym}{\ensuremath{Sym}}

\DeclareMathOperator{\Hom}{\ensuremath{Hom}}

\DeclareMathOperator{\rank}{\ensuremath{rank}}

\newtheorem{innercustomcon}{Conjecture}
\newenvironment{customcon}[1]
  {\renewcommand\theinnercustomcon{#1}\innercustomcon}
  {\endinnercustomcon}

\theoremstyle{plain}
\newtheorem{theorem}{Theorem}[section]
\theoremstyle{plain}
\newtheorem{lemma}[theorem]{Lemma}
\theoremstyle{plain}
\newtheorem{proposition}[theorem]{Proposition}
\theoremstyle{plain}

\newtheorem*{proposition*}{Proposition}
\theoremstyle{plain}
\theoremstyle{plain}

\theoremstyle{definition}

\newtheorem{remark}[theorem]{Remark}

\numberwithin{equation}{section}

\makeatletter
\@namedef{subjclassname@2020}{%
  \textup{2020} Mathematics Subject Classification}
\makeatother
\theoremstyle{plain}

\begin{document}
\title{Kov\'acs' conjecture on characterization of projective space and hyperquadrics}

\author{Soham Ghosh}
\address{Department of Mathematics, University of Washington, Seattle, WA 98195, USA}
\email{\tt soham13@uw.edu }

\keywords{Positivity of tangent bundle, projective space, hyperquadrics, ample vector bundles}

\subjclass[2020]{14M20, 14J45}
\date{\today}

\begin{abstract}
We prove Kovács' conjecture that claims that if the $p^{th}$ exterior power of the tangent bundle of a smooth complex projective variety contains the $p^{th}$ exterior power of an ample vector bundle then the variety is either projective space or the $p$-dimensional quadric hypersurface. We also prove a similar characterization involving symmetric powers instead of exterior powers. This provides a common generalization of Mori, Wahl, Cho--Sato, Andreatta--Wiśniewski, Kobayashi--Ochiai, and Araujo--Druel--Kovács type characterizations of such varieties.
\end{abstract}
\maketitle
%}}}

\section{Introduction}
The problem of classifying projective spaces and hyperquadrics has been extensively studied in literature. The goal of this short note is to give a new characterization of them, and thereby unify several of the existing characterizations based on positivity of the tangent bundle and its exterior powers. The first result in this direction was Mori's resolution of the Hartshorne conjecture in \cite{Mor79}, which classified projective spaces as the only smooth complex projective varieties admitting ample tangent bundle. Around the same time, Siu and Yau \cite{SY80} also proved Frankel's conjecture \cite{Fra61} (which is an analogoue of Hartshorne's conjecture in complex differential geometry) using harmonic maps. The main result of our paper is the following. 

\begin{theorem}\label{Theorem:Main}
Let $X$ be a smooth complex projective variety of dimension $n$. If there exists an ample vector bundle $\mathcal{E}$ of rank $r$ on $X$ and a positive integer $p\leq n$ such that $\bigwedge^p\mathcal{E}\subseteq\bigwedge^pT_X$, then either $X\cong \mathbb P^n$, or $p=n$ and $(X, \mathcal{E})\cong (Q_p, \mathcal{O}_{Q_p}(1)^{\oplus r})$ where $Q_p\subseteq \mathbb{P}^{p+1}$ is a smooth quadric hypersurface in $\mathbb{P}^{p+1}$. Furthermore, if $X\cong \mathbb{P}^n$, then $\det\mathcal{E}\cong\mathcal{O}_{\mathbb{P}^n}(l)$ with $l=r$ or $r+1$ and:
\begin{enumerate}[leftmargin=*]
\item If $r<n$, then $\mathcal{E}\cong \mathcal{O}_{\mathbb{P}^n}(1)^{\oplus r}$.
\item If $r=n$ and $l=n$, then $\mathcal{E}\cong \mathcal{O}_{\mathbb{P}^n}(1)^{\oplus n}$.
\item If $r=n$ and $l=n+1$:
\begin{enumerate}
\item If $1\leq p\leq n-1$, then $\mathcal{E}\cong T_{\mathbb{P}^n}$.
\item If $p=n$, then either
$\mathcal{E}\cong T_{\mathbb{P}^n}$ or $\mathcal{E}\cong\mathcal{O}_{\mathbb{P}^n}(2)\oplus \mathcal{O}_{\mathbb{P}^n}(1)^{\oplus(n-1)}$.
\end{enumerate}
\end{enumerate}
\end{theorem}

Theorem~\ref{Theorem:Main} affirmatively answers the following conjecture of Kovács \cite{Ross11}[Conjecture 1.2].

\begin{customcon}{I}[Kovács]\label{KovacsCon}
Let $X$ be a smooth complex projective variety of dimension $n$. If there exists an ample vector bundle $\mathcal{E}$ of rank $r$ on $X$ and a positive integer $p\leq n$ such that $\bigwedge^p\mathcal{E}\subseteq\bigwedge^pT_X$, then either $X\cong\mathbb{P}^n$ or $p=n$ and $X\cong Q_p\subseteq\mathbb{P}^{p+1}$, where $Q_p$ is a smooth quadric hypersurface in $\mathbb{P}^{p+1}$. 
\end{customcon}

Conjecture~\ref{KovacsCon} was proven in \cite{Ross11} for varieties with Picard number $1$ (\cite{Ross11}[Theorem~1.1]) and for $n=2$ (\cite{Ross11}[Corollary~1.10]). Furthermore, \cite{Liu23}[Theorem~1.15] proves Conjecture~\ref{KovacsCon} for $n\geq 3$, $p=2$ and $\binom{r}{2}\geq n-1$. Using similar technique as for proof of Theorem~\ref{Theorem:Main}, we also prove the following characterization using symmetric powers of vector bundles.

\begin{theorem}\label{Theorem:Main4}
 Let $X$ be a smooth complex projective variety of dimension $n$. If there exists an ample vector bundle $\mathcal{E}$ of rank $r$ on $X$ and a positive integer $p$ such that $\Sym^p\mathcal{E}\subseteq\Sym^pT_X$, then either $X\cong \mathbb P^n$, or $(X, \mathcal{E})\cong (Q_n, \mathcal{O}_{Q_n}(1)^{\oplus r})$ where $Q_n\subseteq \mathbb{P}^{n+1}$ is a smooth quadric hypersurface in $\mathbb{P}^{n+1}$. Furthermore, if $X\cong \mathbb{P}^n$, then $\det\mathcal{E}\cong\mathcal{O}_{\mathbb{P}^n}(l)$ with $l=r$ or $r+1$ and:
\begin{enumerate}[leftmargin=*]
\item If $r<n$, then $\mathcal{E}\cong \mathcal{O}_{\mathbb{P}^n}(1)^{\oplus r}$.
\item If $r=n$ and $l=n$, then $\mathcal{E}\cong \mathcal{O}_{\mathbb{P}^n}(1)^{\oplus n}$.
\item If $r=n$ and $l=n+1$, then $\mathcal{E}\cong T_{\mathbb{P}^n}$.
\end{enumerate}   
\end{theorem}

\subsection{Existing results and background}

We briefly review some of the existing characterizations of projective spaces and hyperquadrics of similar flavor, which serve as motivation for Conjecture~\ref{KovacsCon}. As stated earlier, the first of such results is the following 1979 result of Mori.

\begin{theorem}[\cite{Mor79}]\label{Thm:Mor79}
    Let $X$ be a smooth complex projective variety of dimension $n$ with ample tangent bundle. Then $X\cong \mathbb{P}^n$. 
\end{theorem}

This was followed by a related result of Wahl in 1983 (which was also obtained by Druel in 2004, using different methods).

\begin{theorem}[\cite{Wah83}, \cite{Dru04}]\label{Thm:WahDru}
    Let $X$ be a smooth complex projective variety of dimension $n$, and assume that the tangent bundle $T_X$ contains an ample line bundle $\mathcal{L}$. Then $(X,\mathcal{L})\cong (\mathbb{P}^n, \mathcal{O}_{\mathbb{P}^n}(1))$ or $(X,\mathcal{L})\cong (\mathbb{P}^1, \mathcal{O}_{\mathbb{P}^1}(2))$.
\end{theorem}

In 1998, the above result was generalized by \cite{CP98} from line bundle $\mathcal{L}$ to ample vector bundles $\mathcal{E}$ of rank $r=n, n-1,$ and $n-2$. Interpolating Theorem~\ref{Thm:Mor79} and Theorem~\ref{Thm:WahDru}, Andreatta and Wiśniewski obtained the following characterization in 2001.

\begin{theorem}[\cite{AW01}]\label{Thm:AW01}
    Let $X$ be a smooth complex projective variety of dimension $n$ and assume that the tangent bundle $T_X$ contains an ample vector bundle $\mathcal{E}$ of rank $r$. Then either $(X,\mathcal{E})\cong (\mathbb{P}^n, \mathcal{O}_{\mathbb{P}^n}(1)^{\oplus r})$ or, $r=n$ and $(X,\mathcal{E})\cong (\mathbb{P}^n, T_{\mathbb{P}^n})$.
\end{theorem}

Using the theory of variety of minimal rational tangents, Araujo provided a unified geometric approach to Theorems~\ref{Thm:Mor79}--~\ref{Thm:AW01} in \cite{Ara06}. The following characterization for projective spaces and higher dimensional quadric hypersurfaces was proven in 1973 by Kobayashi and Ochiai.

\begin{theorem}[\cite{KO73}]\label{Thm:KO73}
    Let $X$ be a $n$-dimensional compact complex manifold with ample line bundle $\mathcal{L}$. If $c_1(X)\geq (n+1)c_1(\mathcal{L})$ then $X\cong \mathbb{P}^n$. If $c_1(X)=nc_1(\mathcal{L})$ then $X\cong Q_n$, where $Q_n\subseteq \mathbb{P}^{n+1}$ is a hyperquadric.
\end{theorem}

In 1995, Cho and Sato obtained the following characterization using the second exterior power of the tangent bundle (over arbitrary characteristic).

\begin{theorem}[\cite{ChoSato95}]\label{Thm:ChoSato95}
    Let $X$ be a $n$-dimensional smooth projective variety over an algebraically closed field of arbitrary characteristic such that $\bigwedge^2 T_X$ is ample.
    \begin{enumerate}[leftmargin=*]
        \item if $n\geq 5$, then $X\cong \mathbb{P}^n$ or $Q_n\subseteq \mathbb{P}^n$.
        \item if characteristic is $0$, then (1) holds for $n\geq 3$.
    \end{enumerate}
\end{theorem}

Extending along the above line of results, Araujo-Druel-Kovács proved the following conjecture of Beauville \cite{Bea00}.

\begin{theorem}[\cite{ADK08}]\label{Thm:ADK08}
    Let $X$ be a smooth comlex projective variety of dimension $n$, and let $\mathcal{L}$ be an ample line bundle on $X$. If $H^0(X, \bigwedge^pT_X\otimes\mathcal{L}^{-p})\neq 0$ for some integer $p>0$, then either $(X, \mathcal{L})\cong (\mathbb{P}^n, \mathcal{O}_{\mathbb{P}^n}(1))$, or $p=n$ and $(X,\mathcal{L})\cong (Q_p, \mathcal{O}_{Q_p}(1))$, where $Q_p\subseteq \mathbb{P}^{p+1}$ is a smooth hyperquadric.
\end{theorem}

Despite their similar nature, no one of Theorems~\ref{Thm:Mor79}--~\ref{Thm:ADK08} imply all the other. From this perspective, Conjecture~\ref{KovacsCon} is interesting, since it implies all of the above: Theorem~\ref{Thm:Mor79} follows by taking $p=1$, $\mathcal{E}=T_X$, Theorem~\ref{Thm:WahDru} by $p=1, \ r=1$, Theorem~\ref{Thm:AW01} by $p=1$, Theorem~\ref{Thm:ChoSato95} (for characteristic $0$) by $p=2$ and $\mathcal{E}=T_X$, Theorem~\ref{Thm:ADK08} by taking $\mathcal{E}=\mathcal{L}^{\oplus p}$, and Theorem~\ref{Thm:KO73} by taking $\mathcal{E}=\mathcal{L}^{\oplus n}$ and $\mathcal{L}^{\oplus n-1}\oplus \mathcal{L}^{\otimes 2}$. Indeed, in \cite{Ross11}, Ross proves Conjecture~\ref{KovacsCon} for varieties with Picard number $1$ and also for $n=2$. More recently in \cite{Liu23}, Liu proved the conjecture in higher dimensions for $p=2$ under the rank constraint $\binom{r}{2}\geq n-1$. 

A parallel line of work replaces exterior powers of tangent bundle by tensor powers. In \cite{DP}, Druel and Paris proved that if $T_X^{\otimes r}$ contains the determinant of an ample vector bundle, then $X$ must be $\mathbb{P}^n$ or $Q_n$. More precisely:
\begin{theorem}[{\cite{DP}[Theorem~B]}]\label{thm:DP}
Let $X$ be a smooth complex projective variety of dimension $n$ and $\mathcal{E}$ an ample vector bundle
of rank $r>1$ on $X$. If $H^0\left(X,\ T_X^{\otimes r}\otimes \det(\mathcal{E})^{-1}\right)\neq 0$, then $(X,\det\mathcal{E})\cong (\mathbb P^n, \mathcal{O}_{\mathbb{P}^n}(l))$ with $r\leq l\leq r(n+1)/n$ or $(X, \mathcal{E})\cong (Q_n,\mathcal{O}_{Q_n}(1)^{\oplus r})$ and $r=2i+nj$ with $i,j\gg 0$.
\end{theorem}

Recent results have focused on weakening the positivity assumption from ample to nef or strictly nef, at the cost of rigidity of the structure. Notably, Li--Ou--Yang \cite{LOY19} obtained $X\cong \mathbb{P}^n$ or $Q_n$, assuming $T_X$ is strictly nef, or $n\geq 3$ and $\bigwedge^2T_X$ is strictly nef. Recently, Shao \cite{Shao24} has also conjectured a characterization of projective space and hyperquadrics using cohomology of twists of symmetric powers of tangent bundle and has proven it for varieties with Picard number $1$.

\subsection*{Acknowledgment}
The author would like to thank Giovanni Inchiostro for initial discussions that led the author to the paper \cite{ADK08}. He would also like to thank Sándor Kovács for suggesting the problem and reading an earlier draft of the paper.

\section{Proof of Theorem~\ref{Theorem:Main} and Theorem~\ref{Theorem:Main4}}

\subsection{Characterizing the variety}
In this section, we prove weaker versions of Theorem~\ref{Theorem:Main} and Theorem~\ref{Theorem:Main4}, namely, we obtain a weaker characterization of $\mathcal{E}$, when $X\cong\mathbb{P}^n$. Let $X$ be a smooth complex projective variety of dimension $n$ and let $\mathcal E$ be ample vector bundle on $X$ of rank $r$. For any positive integer integer $1\leq p\leq r$. For the remainder of the paper, set
\[
f := \rank\left(\bigwedge^p \mathcal E\right)=\binom{r}{p},
\
a := \binom{r-1}{p-1}; \qquad f':=\rank\left(\Sym^p\mathcal{E}\right)=\binom{p+r-1}{p}, \ a':= \binom{p+r-1}{r}.
\]
We have the binomial identities 
\begin{equation}\label{Eqn:0}
    p\binom{r}{p} = r\binom{r-1}{p-1}, \text{ i.e.},\ pf = ra;\qquad  p\binom{p+r-1}{p}=r\binom{p+r-1}{r}, \text{ i.e.}, \ pf'=ra'.
\end{equation}

We first prove a linear algebraic fact.

\begin{lemma}\label{Lemma:detpower}
    For any vector bundle $\mathcal{E}$ of rank $r$ and $p\leq r$, we have $\det\left(\bigwedge^p\mathcal E\right)\ \cong\ \det(\mathcal E)^{\otimes a}$, where $a=\binom{r-1}{p-1}$ and $\det\left(\Sym^p\mathcal E\right)\ \cong\ \det(\mathcal E)^{\otimes a'}$, where $a'=\binom{p+r-1}{r}$.
\end{lemma} 
\begin{proof}
Note that $\det(\bigwedge^p(-))$ and $\det(\Sym^p(-))$ are $1$-dimensional representations of the general linear group and hence $\det(\bigwedge^p(-))=\det(-)^{\otimes b}$ and $\det(\Sym^p(-))=\det(-)^{\otimes b'}$ for some integers $b, b'$. Fiberwise, a scalar $\lambda$ acts on $\bigwedge^p\mathcal{E}$ and $\Sym^p\mathcal{E}$ as multiplication by $\lambda^p$. Therefore, fiberwise $\lambda$ acts on $\det\bigwedge^p\mathcal{E}$ by multiplication by $\lambda^{pf}=\lambda^{ra}$ by \eqref{Eqn:0}, which implies $b=a$. Similarly, fiberwise $\lambda$ acts on $\det\Sym^p\mathcal{E}$ by multiplication by $\lambda^{pf'}=\lambda^{ra'}$ by \eqref{Eqn:0}, which implies $b'=a'$.

We note that alternatively, one can take a trivializing cover for $\mathcal{E}$ on $X$ and apply the classical Sylvester-Franke theorem in multilinear algebra \cite{Price47}[\S6] to $\det(\bigwedge^p\mathcal{E})$ and $\det(\Sym^p\mathcal{E})$ at the level of transition functions to obtain the desired identity.
\end{proof}

\begin{lemma}\label{lem:wedge-to-tensor-section}
If there is a subbundle inclusion $\bigwedge^p\mathcal E\subseteq \bigwedge^pT_X$ for any positive integer $p\leq r$, then
\begin{equation}\label{Equation:exterior}
H^0\left(X,\ T_X^{\otimes pf}\otimes \det(\mathcal E)^{-a}\right)\neq 0.
\end{equation}
Similarly, if there is a subbundle inclusion $\Sym^p\mathcal{E}\subseteq\Sym^pT_X$ for any positive integer $p$, then
\begin{equation}\label{Equation:symmetric}
H^0\left(X,\ T_X^{\otimes pf'}\otimes \det(\mathcal E)^{-a'}\right)\neq 0.
\end{equation}
\end{lemma}

\begin{proof}
Let $\iota:\bigwedge^p\mathcal E\hookrightarrow \bigwedge^pT_X$ be the inclusion. Taking the induced map on the $f^{th}$ exterior powers, gives a nonzero morphism
\begin{equation}\label{Eqn:1}
\wedge^f\iota:\det\left(\bigwedge^p\mathcal E\right)\ \longrightarrow\ \bigwedge^f\left(\bigwedge^p T_X\right).
\end{equation}
For any vector bundle $\mathcal{V}$ of rank $r$, there are canonical antisymmetrization maps $\alpha_k:\bigwedge^k\mathcal{V}\rightarrow\mathcal{V}^{\otimes k}$ for $1\leq k\leq r$, given on sections by $\alpha_k(v_1\wedge v_2\wedge\dots\wedge v_k)=\sum_{\sigma\in\Sigma_k}\operatorname{sign}(\sigma) v_{\sigma(1)}\otimes v_{\sigma(2)}\otimes\dots\otimes v_{\sigma(k)}$, where $\Sigma_k$ is the symmetric group on $k$ letters. In characteristic $0$ (or as long as characteristic of the base field is coprime to $k!$), these antisymmetrization maps for vector bundles are injective. The $f^{th}$ antisymmetrization map for $\mathcal{V}=\bigwedge^pT_X$ gives an injective map 
\begin{equation}\label{Eqn:2}
\alpha_f: \bigwedge^{f}\left(\bigwedge^p T_X\right)\hookrightarrow \left(\bigwedge^pT_X\right)^{\otimes f},
\end{equation}
and the $p^{th}$ antisymmetrization map for $\mathcal{V}=T_X$ gives an injective map $\beta_p:\bigwedge^pT_X\hookrightarrow T_X^{\otimes p}$. Taking the induced map on the $f^{th}$ tensor powers we get an injective map
\begin{equation}\label{Eqn:3}
\beta_p^{\otimes f}:\left(\bigwedge^pT_X\right)^{\otimes f}
\hookrightarrow \left(T_X^{\otimes p}\right)^{\otimes f}=T_X^{\otimes pf}.
\end{equation}
Composing all the maps \eqref{Eqn:1}, \eqref{Eqn:2} and \eqref{Eqn:3} yields a non-zero (in fact injective) morphism:
\begin{equation}\label{Eqn:4}
    \det\left(\bigwedge^p\mathcal E\right)\xrightarrow{\wedge^f\iota}\bigwedge^{f}\left(\bigwedge^p T_X\right)\xrightarrow{\alpha_f} \left(\bigwedge^pT_X\right)^{\otimes f}
\xrightarrow{\beta_p^{\otimes f}} T_X^{\otimes pf}.
\end{equation}
Thus, we obtain a nonzero section in $H^0\bigl(X,\ T_X^{\otimes pf}\otimes \det(\bigwedge^p\mathcal E)^{-1}\bigr)$ and proves \eqref{Equation:exterior} by Lemma~\ref{Lemma:detpower}.\\

Similarly, if $\iota':\Sym^p\mathcal{E}\hookrightarrow\Sym^pT_X$ is the inclusion, taking the induced map on $(f')^{th}$ exterior powers gives the injection $\wedge^{f'}\iota':\det(\Sym^p\mathcal{E})\hookrightarrow \bigwedge^{f'}(\Sym^pT_X)$. Using the antisymmetrization map, we have an injective map $\alpha'_{f'}:\bigwedge^{f'}(\Sym^pT_X)\hookrightarrow(\Sym^pT_X)^{\otimes f'}$. Analogous to the antisymmetrization map, for any vector bundle $\mathcal{V}$ of rank $r$, there are canonical symmetrization maps $\gamma_k:\Sym^k\mathcal{V}\rightarrow\mathcal{V}^{\otimes k}$ for any $k\geq 1$, given on sections by $\gamma_k(v_1\cdot v_2\cdots v_k)=\sum_{\sigma\in\Sigma_k}v_{\sigma(1)}\otimes v_{\sigma(2)}\otimes\cdots v_{\sigma(k)}$. When $k!$ is invertible in the base field, the symmetrization map is injective. The $p^{th}$ symmetrization map for $\mathcal{V}=T_X$ therefore gives an injective map $\gamma_{p}: \Sym^pT_X\rightarrow T_X^{\otimes p}$, which therefore induces an injective map $\gamma_p^{\otimes f'}:(\Sym^pT_x)^{\otimes f'}\hookrightarrow T_X^{\otimes pf'}$. Thus, the composition
\begin{equation}\label{Eqn:4sym}
    \det\left(\Sym^p\mathcal E\right)\xrightarrow{\wedge^{f'}\iota'}\bigwedge^{f'}\left(\Sym^p T_X\right)\xrightarrow{\alpha'_{f'}} \left(\Sym^pT_X\right)^{\otimes f'}
\xrightarrow{\gamma_p^{\otimes f'}} T_X^{\otimes pf'}
\end{equation}
yields a non-zero section in $H^0\bigl(X,\ T_X^{\otimes pf'}\otimes \det(\Sym^p\mathcal E)^{-1}\bigr)$ and proves \eqref{Equation:symmetric} by Lemma~\ref{Lemma:detpower}.
\end{proof}

\begin{theorem}\label{Theorem:Main2}
Assume $\bigwedge^p\mathcal E\subseteq \bigwedge^pT_X$ for an ample rank $r$ vector bundle  $\mathcal E$ on $X$, and some positive integer $p\leq r$. Then $(X, \det\mathcal{E})\cong (\mathbb P^n, \mathcal{O}_{\mathbb{P}^n}(l))$ with $r\leq l\leq r(n+1)/n$, or $p=n$ and $(X, \mathcal{E})\cong (Q_p, \mathcal{O}_{Q_p}(1)^{\oplus r})$ where $Q_p\subseteq \mathbb{P}^{p+1}$ is a smooth hyperquadric. 
\end{theorem}

\begin{proof}
Consider the ample vector bundle $\mathcal F := \mathcal E^{\oplus a}$ on $X$. Then $\rank(\mathcal F)=ra$ and $\det(\mathcal F)=\det(\mathcal E)^{\otimes a}$. But we have $pf=ra=\rank(\mathcal{F})$ by \eqref{Eqn:0}, and Lemma~\ref{lem:wedge-to-tensor-section} gives us
\[
H^0\left(X,\ T_X^{\otimes pf}\otimes \det(\mathcal F)^{-1}\right)\neq 0.
\]
Now by Druel-Paris Theorem~\ref{thm:DP}, we have two possibilities:
\begin{enumerate}[leftmargin=*]
    \item  $(X, \det\mathcal{F})\cong (\mathbb{P}^n, \mathcal{O}_{\mathbb{P}^n}(l))$ with $\rank\mathcal{F}\leq l\leq \rank\mathcal{F}(n+1)/n $. But $\det\mathcal{F}=\det \mathcal{E}^{\oplus a}=(\det\mathcal{E})^{\otimes a}\cong \mathcal{O}_{\mathbb{P}^n}(l)$, with $ra\leq l\leq ra(n+1)/n$. Therefore, $\det\mathcal{E}\cong \mathcal{O}_{\mathbb{P}^n}(l)$ with $r\leq l\leq r(n+1)/n$ where $r=\rank\mathcal{E}$.
    \item $(X,\mathcal{F})\cong (Q_n, \mathcal{O}_{Q_n}(1)^{\oplus k})$ where $k=2i+nj$ with $i,j\gg 0$. But $k=\rank\mathcal{F}=ra$, where $r=\rank\mathcal{E}$. Thus, 
    \begin{equation}\label{Eqn:5}
    \mathcal{E}\cong\mathcal{O}_{Q_n}(1)^{\oplus r} \text{ and } \bigwedge^p\mathcal{E}\cong\bigwedge^p\mathcal{O}_{Q_n}(1)^{\oplus r}\cong\mathcal{O}_{Q_n}(p)^{\oplus \binom{r}{p}}.
    \end{equation} 
    Furthermore, by perfect pairing for vector bundles, we have \begin{equation}\label{Eqn:6}
    \bigwedge^pT_{Q_n}\cong (\bigwedge^{n-p}T_{Q_n})^\vee\otimes\bigwedge^nT_{Q_n}\cong\Omega^{n-p}_{Q_n}\otimes \omega^{-1}_{Q_n}\cong \Omega^{n-p}_{Q_n}\otimes\mathcal{O}_{Q_n}(n), 
    \end{equation}
    since $\omega_{Q_n}\cong\mathcal{O}_{Q_n}(-n)$ by the adjunction formula for $Q_n\subseteq \mathbb{P}^{n+1}$. Thus, the inclusion $\bigwedge^p\mathcal{E}\cong \mathcal{O}_{Q_n}(p)^{\oplus \binom{r}{p}}\subseteq\bigwedge^pT_{Q_n}$ gives a non-zero map $\mathcal{O}_{Q_n}(p)\rightarrow\bigwedge^pT_{Q_n}$ and hence by \eqref{Eqn:6}, a non-zero element of 
    \begin{equation}H^0(Q_n,\ \bigwedge^pT_{Q_n}\otimes\mathcal{O}_{Q_n}(-p))=H^0(Q_n,\ \Omega^{n-p}_{Q_n}(n-p)),
    \end{equation}
    implying that $H^0(Q^n,\ \Omega^{n-p}_{Q^n}(n-p))\neq 0$. We will now use the following result of Snow on the cohomology of twisted holomorphic forms on quadric hypersurfaces.
    \begin{theorem}[{\cite{Snow86}[\S4]}]\label{Theorem:Snow}
        Let $X$ be a non-singular quadric hypersurface of dimension $n$.
        \begin{enumerate}
            \item If $-n+q\leq k\leq q$, and $k\neq 0, \ -n+2q$, then $H^i(X, \Omega^q(k))=0$ for all $i$.
            \item $H^i(X, \Omega_X^q)\neq 0$ if and only if $i=q$.
            \item $H^i(X, \Omega^q(-n+2q))\neq 0$ if and only if $i=n-q$.
            \item If $k>q$, then $H^i(X, \Omega^q(k))\neq 0$ if and only if $i=0$.
            \item If $k<-n+q$, then $H^i(X, \Omega^q(k))\neq 0$ if and only if $i=n$.
        \end{enumerate}
    \end{theorem}
     By taking $k=q=n-p$ and $i=0$ in Theorem~\ref{Theorem:Snow}(a), we have $H^0(Q_n,\ \Omega^{n-p}_{Q_n}(n-p))=0$, unless $n-p=0$ or $n-p=-n+2(n-p)$ (equivalently, $p=0$). Since we found  $H^0(Q_n,\ \Omega^{n-p}_{Q_n}(n-p))\neq 0$, we must have $p=n$ or $p=0$. But $p>0$ by hypothesis, so we therefore have $p=n$.
\end{enumerate}
\end{proof}

\begin{theorem}\label{Theorem:Main2symmetric}
Assume $\Sym^p\mathcal E\subseteq \Sym^pT_X$ for an ample rank $r$ vector bundle  $\mathcal E$ on $X$, and some positive integer $p$. Then $(X, \det\mathcal{E})\cong (\mathbb P^n, \mathcal{O}_{\mathbb{P}^n}(l))$ with $r\leq l\leq r(n+1)/n$, or  $(X, \mathcal{E})\cong (Q_n, \mathcal{O}_{Q_n}(1)^{\oplus r})$ where $Q_n\subseteq \mathbb{P}^{n+1}$ is a smooth hyperquadric. 
\end{theorem}

\begin{proof}
Consider the ample vector bundle $\mathcal F := \mathcal E^{\oplus a'}$ on $X$. Then $\rank(\mathcal F)=ra'$ and $\det(\mathcal F)=\det(\mathcal E)^{\otimes a'}$. But we have $pf'=ra'=\rank(\mathcal{F})$ by \eqref{Eqn:0}, and Lemma~\ref{lem:wedge-to-tensor-section} gives us
\[
H^0\left(X,\ T_X^{\otimes pf'}\otimes \det(\mathcal F)^{-1}\right)\neq 0.
\]
Thus, Druel-Paris Theorem~\ref{thm:DP}, gives us the two possibilities:
\begin{enumerate}[leftmargin=*]
    \item  $(X, \det\mathcal{F})\cong (\mathbb{P}^n, \mathcal{O}_{\mathbb{P}^n}(l))$ with $\rank\mathcal{F}\leq l\leq \rank\mathcal{F}(n+1)/n $. But $\det\mathcal{F}=\det \mathcal{E}^{\oplus a'}=(\det\mathcal{E})^{\otimes a'}\cong \mathcal{O}_{\mathbb{P}^n}(l)$, with $ra'\leq l\leq ra'(n+1)/n$. Therefore, $\det\mathcal{E}\cong \mathcal{O}_{\mathbb{P}^n}(l)$ with $r\leq l\leq r(n+1)/n$ where $r=\rank\mathcal{E}$.
    \item $(X,\mathcal{F})\cong (Q_n, \mathcal{O}_{Q_n}(1)^{\oplus k})$ where $k=2i+nj$ with $i,j\gg 0$. But $k=\rank\mathcal{F}=ra'$, where $r=\rank\mathcal{E}$. Thus, $\mathcal{E}=\mathcal{O}_{Q_n}(1)^{\oplus r}$.
    \end{enumerate}
    \end{proof}

\subsection{Characterizing $\mathcal{E}$ when $X\cong\mathbb{P}^n$} In this section we characterize ample vector bundles $\mathcal{E}$ of rank $r$ on $\mathbb{P}^n$ satisfying either of the following two conditions:
\begin{itemize}[leftmargin=*]
    \item \textbf{Condition A}: $\bigwedge^p\mathcal{E}\subseteq \bigwedge^p T_{\mathbb{P}^n}$ and $\det\mathcal{E}\cong \mathcal{O}_{\mathbb{P}^n}(l)$ with $r\leq l\leq \frac{r(n+1)}{n}$ or,
    \item \textbf{Condition B}: $\Sym^p\mathcal{E}\subseteq\Sym^pT_{\mathbb{P}^n}$ and $\det\mathcal{E}\cong \mathcal{O}_{\mathbb{P}^n}(l)$ with $r\leq l\leq \frac{r(n+1)}{n}$.
\end{itemize}
Clearly, we have $0<r\leq n$ in either situations as $\rank(\bigwedge^p\mathcal{V})$ and $\rank(\Sym^p\mathcal{V})$ are both increasing functions of $\rank(\mathcal{V})$. In Condition A, we also have the constraint $0<p\leq n$, whereas in Condition B, $p$ is any positive integer. First note the following well-known splitting type of the tangent bundle $T_{\mathbb{P}^n}$ on any line $L\subset \mathbb{P}^n$ (see for example, \cite{OSS88}[Example, Pg~14]).

\begin{lemma}\label{lem:TL}
Let $L\cong \mathbb{P}^1$ be a line in $\mathbb{P}^n$. Then $T_{\mathbb{P}^n}|_L \ \cong\ \mathcal{O}_L(2)\ \oplus\ \mathcal{O}_L(1)^{\oplus (n-1)}$.
\end{lemma}

Consequently, one obtains the following splitting type of exterior powers of $T_{\mathbb{P}^n}$.

\begin{lemma}\label{lem:wedgeTL}
For $1\leq p\leq n$ and a line $L\cong \mathbb{P}^1$ in $\mathbb{P}^n$, one has
\[
\left(\bigwedge^pT_{\mathbb{P}^n}\right)\Bigg|_{L}\cong \bigwedge^p(T_{\mathbb{P}^n}|_L)\ \cong\
\mathcal{O}_L(p+1)^{\oplus \binom{n-1}{p-1}}\ \oplus\ \mathcal{O}_L(p)^{\oplus \binom{n-1}{p}}.
\]
In particular, every line bundle in the splitting type of $\bigwedge^pT_{\mathbb{P}^n}$ has degree either $p$ or $p+1$, and the maximal degree occurring is $p+1$.
\end{lemma}

\begin{proof}
By Lemma~\ref{lem:TL},
$T_{\mathbb{P}^n}|_L\cong \mathcal{O}_L(2)\oplus \mathcal{O}_L(1)^{\oplus(n-1)}$. Since direct sum and exterior powers commute, taking the $p^{th}$ wedge and expanding
gives two types of terms: $(i)$ those where we choose the $\mathcal{O}_L(2)$-summand (and hence
choose $p-1$ copies of $\mathcal{O}_L(1)$), which contribute $\mathcal{O}_L(2+(p-1))=\mathcal{O}_L(p+1)$ with multiplicity
$\binom{n-1}{p-1}$, and $(ii)$ those where we choose only $\mathcal{O}_L(1)$-summands contributing $\mathcal{O}_L(p)$
with multiplicity $\binom{n-1}{p}$.
\end{proof}

Analogously, we have the following for symmetric powers of $T_{\mathbb{P}^n}$.

\begin{lemma}\label{lem:symTL}
For $p>0$ and a line $L\cong \mathbb{P}^1$ in $\mathbb{P}^n$, one has
\[
\left(\Sym^pT_{\mathbb{P}^n}\right)|_{L}\cong \Sym^p(T_{\mathbb{P}^n}|_L)\ \cong\
\bigoplus_{\substack{m_1+m_2+\dots+m_n=p\\ m_i\geq0}}\mathcal{O}_L(2m_1+m_2+\dots+m_n).
\]
In particular, every line bundle in the splitting type of $\Sym^pT_{\mathbb{P}^n}$ has degree $p+m_1$ for some $0\leq m_1\leq p$, and the maximal degree occurring is $2p$, with multiplicity $1$.
\end{lemma}

\begin{proof}
By Lemma~\ref{lem:TL},
$T_{\mathbb{P}^n}|_L\cong \mathcal{O}_L(2)\oplus \mathcal{O}_L(1)^{\oplus(n-1)}$. The splitting type of $\Sym^pT_{\mathbb{P}^n}$ follows from noting that direct sum and symmetric powers commute. Clearly the maximal degree line bundle in the splitting can only occur when $m_1=p$ and $m_i=0$ for all $i\geq 2$, which gives degree $2p$ with multiplicity $1$. 
\end{proof}

We now find the splitting type of ample vector bundle $\mathcal{E}$ on $\mathbb{P}^n$ satisfying Condition~A.

\begin{proposition}\label{prop:splittype}
Let $L\cong \mathbb{P}^1$ be any line. Let $\mathcal{E}$ be an ample vector bundle of rank $r$ on $\mathbb{P}^n$ satisfying Condition A above. Then $l:=\deg\det\mathcal{E}=r$ or $r+1$ and:
\begin{enumerate}[leftmargin=*]
\item if $l=r$ then $\mathcal{E}|_L\cong \mathcal{O}_L(1)^{\oplus r}$ for every line $L$;
\item if $l=r+1$ then $\mathcal{E}|_L\cong \mathcal{O}_L(2)\oplus \mathcal{O}_L(1)^{\oplus (r-1)}$ for every line $L$.
\end{enumerate}
In either case $\mathcal{E}$ is a \emph{uniform} bundle (i.e., the splitting type on lines is independent of $L$).
\end{proposition}

\begin{proof}
Let $\mathcal{E}|_L \cong \bigoplus_{i=1}^r \mathcal{O}_L(a_i)$ with $a_1\geq a_2\geq \cdots \geq a_r$. Since $\mathcal{E}$ is ample, so is the restriction $\mathcal{E}|_L$ to any line $L\subset\mathbb{P}^n$. By \cite{Hartshorne66}[\S 7, Example 1] and \cite{Hartshorne66}[Proposition~7.1], we have $a_r\geq 1$. Restrict the given injection $\iota:\bigwedge^p\mathcal{E}\hookrightarrow\bigwedge^pT_{\mathbb{P}^n}$ to a line $L\subset \mathbb{P}^n$ to obtain
\[
\iota|_L:\bigwedge^p(\mathcal{E}|_L)\ \hookrightarrow\ \bigwedge^p(T_{\mathbb{P}^n}|_L).
\]
By Lemma~\ref{lem:wedgeTL}, the target is a direct sum of copies of $\mathcal{O}_L(p+1)$ and $\mathcal{O}_L(p)$.
On the other hand, $\bigwedge^p(\mathcal{E}|_L)$ splits as a direct sum of line bundles $\mathcal{O}_L(a_{i_1}+a_{i_2}+\dots+a_{i_p})$ for subsets $\{i_1,\dots, i_p\}\subseteq\{1,2,\dots, r\}$, and
therefore, $\mathcal{O}_L(a_1+\cdots+a_p)$ is a \emph{maximal} degree summand since $a_1\geq a_2\geq\dots\geq a_r$. Thus, we have an injection $\mathcal{O}_L(a_1+\cdots+a_p)\rightarrow \bigwedge^p(T_{\mathbb{P}^n}|_L)$, and hence by Lemma~\ref{lem:wedgeTL} we have $a_1+\dots+a_p\leq p+1$, since a non-zero map $\mathcal{O}_{\mathbb{P}^1}(d)\rightarrow\bigoplus_{i=1}\mathcal{O}_{\mathbb{P}^1}(d_i)$ exists if and only if $d\leq \max_id_i$. This also implies $a_1\leq 2$ since $a_i\geq a_r\geq 1$ for all $1\leq i\leq r$. Furthermore, only one of the $a_i$'s can equal $2$. This is because if $p\geq 2$, and at least two of the $a_i$'s are $2$, then $a_1+\dots+a_p\geq2+2+(p-2)>p+1$. For $p=1$, we have $\mathcal{E}|_L\hookrightarrow T_{\mathbb{P}^n}|_L=\mathcal{O}_L(2)\oplus \mathcal{O}_L(1)^{\oplus(n-1)}$, but
$\Hom(\mathcal{O}_L(2),\mathcal{O}_L(1))=0$, and so at most one $\mathcal{O}_L(2)$ can inject into $T_{\mathbb{P}^n}|_L$. Hence at most one $a_i$ can equal $2$. Thus $a_1\in\{1,2\}$ and all other $a_i=1$ for $2\leq i\leq r$, proving the stated two possible splitting types on a line. Finally, $\det(\mathcal{E})|_L \cong \det(\mathcal{E}|_L)\cong \mathcal{O}_L(\sum_{i=1}^{r} a_i)$ has degree
\[
\sum_{i=1}^r a_i = \deg(\det\mathcal{E}|_L)=\deg(\mathcal{O}_{\mathbb{P}^n}(l)|_L)=l,
\] 
Thus, $l=\sum_{i=1}^{r}a_i$ equals either $r$ or $r+1$, depending on whether $(a_1,\dots, a_r)=(1,1,\dots, 1)$ or $(2,1,1\dots, 1)$ respectively. This in particular shows that the splitting type is constant on all lines, i.e., $\mathcal{E}$ is uniform.
\end{proof}

A very similar analysis yields the splitting type of an ample vector bundle $\mathcal{E}$ on $\mathbb{P}^n$ satisfying Condition B.

\begin{proposition}\label{prop:splittypesymmetric}
Let $L\cong \mathbb{P}^1$ be any line. Let $\mathcal{E}$ be an ample vector bundle of rank $r$ on $\mathbb{P}^n$ satisfying Condition B above. Then $l:=\deg\det\mathcal{E}=r$ or $r+1$ and:
\begin{enumerate}[leftmargin=*]
\item if $l=r$ then $\mathcal{E}|_L\cong \mathcal{O}_L(1)^{\oplus r}$ for every line $L$;
\item if $l=r+1$ then $\mathcal{E}|_L\cong \mathcal{O}_L(2)\oplus \mathcal{O}_L(1)^{\oplus (r-1)}$ for every line $L$.
\end{enumerate}
In either case $\mathcal{E}$ is a \emph{uniform} bundle (i.e., the splitting type on lines is independent of $L$).
\end{proposition}

\begin{proof}
As in the above proof, $\mathcal{E}|_L \cong \bigoplus_{i=1}^r \mathcal{O}_L(a_i)$ with $a_1\geq a_2\geq \cdots \geq a_r\geq 1$. Restrict the given injection $\iota':\Sym^p\mathcal{E}\hookrightarrow\Sym^pT_{\mathbb{P}^n}$ to a line $L\subset \mathbb{P}^n$ to obtain
\[
\iota'|_L:\Sym^p(\mathcal{E}|_L)\ \hookrightarrow\ \Sym^p(T_{\mathbb{P}^n}|_L).
\]
By Lemma~\ref{lem:symTL}, the target is a direct sum of copies of $\mathcal{O}_L(m_1+p)$, where $0\leq m_i\leq p$. On the other hand, $\Sym^p(\mathcal{E}|_L)$ splits as a direct sum of line bundles $\mathcal{O}_L(a_1m_1+a_2m_2+\dots+a_rm_r)$, where $m_i\geq 0$ and $m_1+\dots+m_r=p$. If $a_1>2$, then one of the summands in $\Sym^p(\mathcal{E}|_L)$ would have degree $pa_1>2p$, which is the maximal degree among line bundles occurring in $\Sym^p(T_{\mathbb{P}^n}|_L)$. This would violate the inclusion $\iota'|_L:\Sym^p(\mathcal{E}|_L)\ \hookrightarrow\ \Sym^p(T_{\mathbb{P}^n}|_L)$ and thus, $a_1\leq 2$. Similarly if $a_1=a_2=2$, then the multiplicity of $\mathcal{O}_L(2p)$ in $\Sym^p(\mathcal{E}|_L)$ would be the number of solutions $m_1+m_2=p$, with $m_i\geq 0$, which is greater than $1$. Since $\mathcal{O}_L(2p)$ occurs with multiplicity $1$ in $\Sym^p(T_{\mathbb{P}^n}|_L)$, we see that $a_2=2$ violates the inclusion $\iota'|_L:\Sym^p(\mathcal{E}|_L)\ \hookrightarrow\ \Sym^p(T_{\mathbb{P}^n}|_L)$. Thus $a_1\in\{1,2\}$ and all other $a_i=1$ for $2\leq i\leq r$, proving the stated two possible splitting types on a line. Finally, as in the proof of Proposition~\ref{prop:splittype}, $\det(\mathcal{E})\cong \mathcal{O}_{\mathbb{P}^n}(\sum_{i=1}^{r} a_i)\cong \mathcal{O}_{\mathbb{P}^n}(l)$ with $l$ equal to $r$ or $r+1$, depending on whether $(a_1,\dots, a_r)=(1,1,\dots, 1)$ or $(2,1,1\dots, 1)$ respectively. This also shows that $\mathcal{E}$ is uniform.
\end{proof}

\begin{remark}
    Propositions~\ref{prop:splittype} and ~\ref{prop:splittypesymmetric} give independent proofs of $l=r$ or $r+1$, which we also obtained in Theorems~\ref{Theorem:Main2} and ~\ref{Theorem:Main2symmetric} from Druel-Paris's bounds in Theorem~\ref{thm:DP}.
\end{remark}

We now note a theorem of Elencwajg--Hirschowitz--Schneider on uniform vector bundles on $\mathbb{P}^n$, which will be useful.

\begin{theorem}[\cite{EHS80}]\label{thm:EHS}
Let $\mathcal{F}$ be a uniform vector bundle of rank $r\leq n$ on $\mathbb{P}^n_{\mathbb{C}}$.
Then either $\mathcal{F}$ is a direct sum of line bundles, or $r=n$ and $\mathcal{F}$ is (up to tensoring by a line bundle)
the tangent bundle $T_{\mathbb{P}^n}$ or the cotangent bundle $\Omega^1_{\mathbb{P}^n}$.
\end{theorem}

We can now prove the main classification result.

\begin{theorem}\label{Theorem:Main3}
Let $\mathcal{E}$ be an ample vector bundle of rank $r$ on $\mathbb{P}^n$ satisfying either Condition (A) or (B). Then $r\leq n$ and:
\begin{enumerate}[leftmargin=*]
\item If $r<n$, then $\mathcal{E}\cong \mathcal{O}_{\mathbb{P}^n}(1)^{\oplus r}$.
\item If $r=n$ and $l=n$, then $\mathcal{E}\cong \mathcal{O}_{\mathbb{P}^n}(1)^{\oplus n}$.
\item If $r=n$, $l=n+1$ and $\mathcal{E}$ satisfies Condition A:
\begin{enumerate}
\item If $1\leq p\leq n-1$, then $\mathcal{E}\cong T_{\mathbb{P}^n}$.
\item If $p=n$, then either
$\mathcal{E}\cong T_{\mathbb{P}^n}$ or $\mathcal{E}\cong\mathcal{O}_{\mathbb{P}^n}(2)\oplus \mathcal{O}_{\mathbb{P}^n}(1)^{\oplus(n-1)}$.
\end{enumerate}
\item If $r=n$, $l=n+1$ and $\mathcal{E}$ satisfies Condition B, then $\mathcal{E}\cong T_{\mathbb{P}^n}$.
\end{enumerate}
\end{theorem}

\begin{proof}
Since $\rank T_{\mathbb{P}^n}=n$, we know $r\leq n$. We will assume $n\geq 2$ as for $n=1$, it is immediate that $\mathcal{E}\cong\mathcal{O}_{\mathbb{P}^1}(1)$ or $T_{\mathbb{P}^1}\cong\mathcal{O}_{\mathbb{P}^1}(2)$. Now we have the following cases.

\smallskip\noindent\emph{Case 1: $r<n$.}
If $r<n$, then $r+1>r(n+1)/n$, and so $r\leq l\leq r(n+1)/n$ implies $l=r$. By Proposition \ref{prop:splittype} and \ref{prop:splittypesymmetric}, the bundle
$\mathcal{E}$ is uniform with splitting type $(1,\dots,1)$ irrespective of whether $\mathcal{E}$ satisfies Condition A or B. Applying Theorem \ref{thm:EHS} (with $r<n$) shows that $\mathcal{E}$ splits as a sum of line bundles, say
$\mathcal{E}\cong \bigoplus_{i=1}^r \mathcal{O}_{\mathbb{P}^n}(a_i)$.
Restricting to any line $L$ and using Proposition \ref{prop:splittype}, each $a_i$ must equal $1$,
hence $\mathcal{E}\cong \mathcal{O}(1)^{\oplus r}$.

\smallskip\noindent\emph{Case 2: $r=n$ and $l=n$.}
Then Proposition \ref{prop:splittype} and \ref{prop:splittypesymmetric} give $\mathcal{E}|_L\cong \mathcal{O}_L(1)^{\oplus n}$ on every line $L$, irrespective of whether $\mathcal{E}$ satisfies Condition A or B. By Theorem \ref{thm:EHS}, $\mathcal{E}$ is either a sum of line bundles or a twist of $T_{\mathbb{P}^n}$ or $\Omega^1_{\mathbb{P}^n}$. But the latter two possibilities are excluded by Lemma~\ref{lem:TL}: $T_{\mathbb{P}^n}|_L \ \cong\ \mathcal{O}_L(2)\ \oplus\ \mathcal{O}_L(1)^{\oplus (n-1)}$ and $\Omega^1_{\mathbb{P}^n}|_L \ \cong\ \mathcal{O}_L(-2)\ \oplus\ \mathcal{O}_L(-1)^{\oplus (n-1)}$. Thus $\mathcal{E}\cong \oplus_i \mathcal{O}(a_i)$, and as in Case 1, we have $\mathcal{E}\cong \mathcal{O}(1)^{\oplus n}$.

\smallskip\noindent\emph{Case 3: $r=n$, $l=n+1$ and $\mathcal{E}$ satisfies Condition A.}
By Proposition \ref{prop:splittype}, $\mathcal{E}$ is uniform with splitting type $(2,1,\dots,1)$ on every line.
By Theorem \ref{thm:EHS} it is either a direct sum of line bundles or a twist of $T_{\mathbb{P}^n}$ or $\Omega^1_{\mathbb{P}^n}$.

\smallskip
\noindent \emph{(a) Assume $1\leq p\leq n-1$.}
If $\mathcal{E}$ were split, the only possibility compatible with the line splitting type is
$\mathcal{E}\cong \mathcal{O}_{\mathbb{P}^n}(2)\oplus \mathcal{O}_{\mathbb{P}^n}(1)^{\oplus(n-1)}$.
Then for $1\leq p\leq n-1$,
\[
\bigwedge^p \mathcal{E}
\ \cong\
\mathcal{O}_{\mathbb{P}^n}(p+1)^{\oplus \binom{n-1}{p-1}}\ \oplus\ \mathcal{O}_{\mathbb{P}^n}(p)^{\oplus \binom{n-1}{p}},
\]
which implies $\bigwedge^p\mathcal{E}$ contains $\mathcal{O}_{\mathbb{P}^n}(p+1)$ as a direct summand and thus, the inclusion $\bigwedge^p\mathcal{E}\subseteq \bigwedge^pT_{\mathbb{P}^n}$ induces a non-zero map $\mathcal{O}_{\mathbb{P}^n}(p+1)\rightarrow\bigwedge^p T_{\mathbb{P}^n}$. Using $\bigwedge^pT_{\mathbb{P}^n} \cong \Omega^{n-p}_{\mathbb{P}^n}\otimes \det(T_{\mathbb{P}^n}) \cong \Omega^{n-p}_{\mathbb{P}^n}(n+1)$, we therefore obtain
\begin{equation}\label{Eqn:7}
\Hom\bigl(\mathcal{O}_{\mathbb{P}^n}(p+1),\ \bigwedge^p T_{\mathbb{P}^n}\bigr)
\cong
H^0\bigl(\mathbb{P}^n,\ \bigwedge^p T_{\mathbb{P}^n}\,(-p-1)\bigr)\cong H^0\bigl(\mathbb{P}^n, \ \Omega^{n-p}_{\mathbb{P}^n}(n-p)\bigr) \neq 0.
\end{equation}
By the well-known Bott formula (\cite{Bott57}), $H^0(\mathbb{P}^n, \Omega^{q}_{\mathbb{P}^n}(k))=0$ when $k\leq q$ and $q>0$, and therefore for $k=q=n-p\geq 1$ implies $H^0(\mathbb{P}^n, \ \Omega_{\mathbb{P}^n}^{n-p}(n-p))=0$. This contradicts \eqref{Eqn:7} and hence, $\mathcal{E}$ cannot be a direct sum of line bundles when $p\leq n-1$. This leaves the possibility of $\mathcal{E}\cong T_{\mathbb{P}^n}(a)$ or $\Omega_{\mathbb{P}^n}(b)$ for some $a,b\in\mathbb{Z}$. But we know $\mathcal{E}|_L\cong \mathcal{O}_L(2)\oplus \mathcal{O}_L(1)^{\oplus(n-1)}$ from Proposition~\ref{prop:splittype}. Comparing this with $T_{\mathbb{P}^n}(a)|_L\cong \mathcal{O}_L(2+a)\ \oplus\ \mathcal{O}_L(1+a)^{\oplus (n-1)}$ and $\Omega_{\mathbb{P}^n}(b)|_L\cong \mathcal{O}_L(-2+b)\ \oplus\ \mathcal{O}_L(-1+b)^{\oplus (n-1)}$ (by Lemma~\ref{lem:TL}), we see that the only possibilities are $\mathcal{E}\cong T_{\mathbb{P}^n}$ or $\mathcal{E}\cong \Omega_{\mathbb{P}^n}(3)$, when $n=2$. But by perfect pairing for vector bundles, $\Omega^1_{\mathbb{P}^2}\cong(\Omega^1_{\mathbb{P}^2})^\vee\otimes\det\Omega^1_{\mathbb{P}^2}\cong T_{\mathbb{P}^2}\otimes\mathcal{O}_{\mathbb{P}^2}(-3)$ implying $T_{\mathbb{P}^2}\cong \Omega^1_{\mathbb{P}^2}(3)$. Thus, we have $\mathcal{E}\cong T_{\mathbb{P}^n}$ always for $n\geq 2$.

\smallskip
\noindent\emph{(b) Assume $p=n$.}
Here $\bigwedge^n\mathcal{E}=\det\mathcal{E}\cong \mathcal{O}_{\mathbb{P}^n}(n+1)=\det(T_{\mathbb{P}^n})=\bigwedge^nT_{\mathbb{P}^n}$, and the hypothesis simply says that there exists an injection $\mathcal{O}_{\mathbb{P}^n}(n+1)\hookrightarrow \mathcal{O}_{\mathbb{P}^n}(n+1)$, which is necessarily an isomorphism. Thus $p=n$ imposes no extra constraint beyond $l=n+1$. In this case all the possibilities given by Theorem \ref{thm:EHS} can occur. Like case \emph{(a)} above, Proposition~\ref{prop:splittype} tells us that the splitting type of $\mathcal{E}$ is $(2,1,\dots, 1)$. Therefore,
\begin{itemize}[leftmargin=*]
    \item If $\mathcal{E}$ is split, then $\mathcal{E}\cong \mathcal{O}_{\mathbb{P}^n}(2)\oplus \mathcal{O}_{\mathbb{P}^n}(1)^{\oplus(n-1)}$.
    \item If $\mathcal{E}\cong T_{\mathbb{P}^n}(a)$ for $a\in\mathbb{Z}$, then by restricting to any line $L$ and comparing splitting types, the only possibility is $a=0$.
    \item If $\mathcal{E}\cong \Omega_{\mathbb{P}^n}(b)$, then similarly by splitting types, the only possibility is $b=3$, when $n=2$. But then, as we saw above, $\Omega_{\mathbb{P}^2}(3)\cong T_{\mathbb{P}^2}$.
\end{itemize}

\smallskip\noindent\emph{Case 4: $r=n$, $l=n+1$ and $\mathcal{E}$ satisfies Condition B.}
By Proposition \ref{prop:splittypesymmetric}, $\mathcal{E}$ is uniform with splitting type $(2,1,\dots,1)$ on every line.
By Theorem \ref{thm:EHS} it is either a direct sum of line bundles or a twist of $T_{\mathbb{P}^n}$ or $\Omega^1_{\mathbb{P}^n}$. If $\mathcal{E}$ were split, the only possibility compatible with the line splitting type is
$\mathcal{E}\cong \mathcal{O}_{\mathbb{P}^n}(2)\oplus \mathcal{O}_{\mathbb{P}^n}(1)^{\oplus(n-1)}$. Then for $p>0$, $\Sym^p\mathcal{E}$ has the summand $\mathcal{O}_{\mathbb{P}^n}(2p)$, which along with the inclusion $\Sym^p\mathcal{E}\subseteq \Sym^pT_{\mathbb{P}^n}$ yields
\begin{equation}\label{Eqn:8}
\Hom\bigl(\mathcal{O}_{\mathbb{P}^n}(2p),\ \Sym^p T_{\mathbb{P}^n}\bigr)
\cong
H^0\bigl(\mathbb{P}^n,\ \Sym^p T_{\mathbb{P}^n}\otimes\mathcal{O}_{\mathbb{P}^n}(-2p)\bigr) \neq 0.
\end{equation}
By \cite{Shao24}[Theorem~5.4], \eqref{Eqn:8} is not possible for any $p>0$ when $n\geq 2$. Hence $\mathcal{E}$ cannot be a direct sum of line bundles. This leaves the possibility of $\mathcal{E}\cong T_{\mathbb{P}^n}(a)$ or $\Omega_{\mathbb{P}^n}(b)$ for some $a,b\in\mathbb{Z}$. But as in Case 3, comparing splitting type yields the only possibility to be $\mathcal{E}\cong T_{\mathbb{P}^n}$.
\end{proof}

Thus, Theorem~\ref{Theorem:Main} follows from Theorems~\ref{Theorem:Main2} and \ref{Theorem:Main3} and Theorem~\ref{Theorem:Main4} follows from Theorems~\ref{Theorem:Main2symmetric} and \ref{Theorem:Main3}.

\end{document}